\documentclass[11pt]{amsart}

\usepackage{amsmath, amsfonts, amssymb,amsthm}
\usepackage{amstext}
\usepackage{mathrsfs}

\usepackage{float,epsf,subfigure}
\usepackage[all,cmtip]{xy}
\usepackage{hyperref}
\usepackage{algorithm}
\usepackage{algorithmic}
\usepackage{mathtools}
\usepackage{float}
\usepackage{bm}
\theoremstyle{plain}
\newtheorem{theorem}{Theorem}[section]

\newtheorem{cor}[theorem]{Corollary}
\newtheorem{def-thm}[theorem]{Definition-Theorem}
\newtheorem{lemma}[theorem]{Lemma}

\newtheorem{defi}[theorem]{Definition}

\theoremstyle{definition}

\def\min{\mathop{\mathrm{min}}}
\def\CC{\mathbb C}

\def\PP{\mathbb P}

\begin{document}
\title[Holomorphic curves ]{Holomorphic curves whose domains are Riemann surfaces}
\author[X.J. Dong]
{Xianjing Dong}

\address{Academy of Mathematics and Systems Sciences \\ Chinese Academy of Sciences \\ Beijing, 100190, P.R. China}
\email{xjdong@amss.ac.cn}


\date{}
\maketitle \thispagestyle{empty} \setcounter{page}{1}

\begin{abstract}
\noindent We establish a defect relation of  holomorphic curves from  a general open Riemann surface into a normal complex projective
variety, with Zariski-dense image intersecting effective
Cartier divisors. 
\end{abstract}

\vskip\baselineskip

\setlength\arraycolsep{2pt}
\medskip

\section{Introduction}

Value distribution  of holomorphic curves has grown into a very rich branch in Nevanlinna theory \cite{Noguchi, ru} since H. Cartan \cite{Cart} established his Second Main Theorem of holomorphic curve from $\mathbb C$ into $\mathbb P^n(\CC)$ intersecting hyperplanes in general position. 
Many well-known results were obtained, referred to Ahlfors \cite{ahlfors}, Nochka \cite{noc, nochka}, 
Noguchi-Winkelmann \cite{Noguchi-W, Noguchi}, Ru \cite{ru,ru00,ru0,ru1,Ru-Sibony}, Shabat \cite{Shabat}, Tiba \cite{tiba} and Yamanoi \cite{yam},  etc. In the   paper, we  would further  develop the 
well-known Ru's result of holomorphic curves by generalizing  the source space $\mathbb C$ to a general open Riemann surface through  Brownian motion  initiated by Carne \cite{carne} and   developed by Atsuji \cite{at,atsuji}.

Let $S$ be an open  Riemann surface. By uniformization theorem,
one could equip $S$ with a complete Hermitian metric $ds^2=2gdzd\overline{z}$ such that
the
   Gauss curvature
$K_S\leq0$  associated to $g,$  here $K_S$ is defined by
$$K_S=-\frac{1}{4}\Delta_S\log g=-\frac{1}{g}\frac{\partial^2\log g}{\partial z\partial\overline z}.$$
Obviously, $(S,g)$ is a complete K\"ahler manifold with  associated K\"ahler form
$\alpha=g\frac{\sqrt{-1}}{2\pi}dz\wedge d\overline{z}.$
Set
\begin{equation}\label{kappa}
  \kappa(t)=\min\big\{K_S(x): x\in \overline{D(t)}\big\}
\end{equation}
which is a non-positive and decreasing continuous function defined on $[0,\infty).$

Fix $o\in S$ as a reference point. Denote by $D(r)$ the geodesic disc centered at $o$ with radius $r,$ and by $\partial D(r)$ the boundary of $D(r).$
By Sard's theorem, $\partial D(r)$ is a submanifold of $S$ for almost all $r>0.$
Also, we denote by $g_r(o,x)$ the Green function of $\Delta_S/2$ with Dirichlet boundary condition and a pole at $o,$ and by $d\pi^r_o(x)$ the harmonic
measure on $\partial D(r)$ with respect to $o.$

Let $$f:S\rightarrow X$$  be a holomorphic curve, where   $X$ is a complex projective variety.  Let us   
first  introduce   Nevanlinna's functions on
  Riemann surfaces which are extensions of  the classical ones on $\mathbb C.$
Let $L\rightarrow X$
be an ample holomorphic line bundle equipped with Hermitian metric $h.$
 We define the \emph{characteristic function} of $f$ with respect to $L$ by
 \begin{eqnarray*}\label{}
   T_{f,L}(r)
   &=& \pi\int_{D(r)}g_r(o,x)f^*c_1(L,h) \\
   &=&-\frac{1}{4}\int_{D(r)}g_r(o,x)\Delta_S\log h\circ f(x)dV(x),
 \end{eqnarray*}
 where $dV(x)$ is the Riemannian volume measure of $S.$ It can be easily known that $T_{f,L}(r)$ is independent
 of the choices of  metrics on $L,$ up to a bounded term.  Since a holomorphic line bundle  on $X$
 can be written as the difference of two  ample holomorphic line bundles, the definition of $T_{f,L}(r)$ can extend to
 an arbitrary holomorphic line bundle.
 For a convenience, we use $T_{f,D}(r)$ to replace $T_{f,L_D}(r)$ for an effective Cartier divisor $D$ on $X.$
 Given an ample effective Cartier divisor $D$ on $X,$
  the Weil function of $D$ is well defined by 
$$\lambda_D(x)=-\log\|s_D(x)\|$$
up to a bounded term, and here  $s_D$ is the canonical section associated to $D.$ Note also that an effective Cartier divisor
 can be written as the difference of two ample effective Cartier divisors, and so the definition for  Weil functions can  extend to  an arbitrary 
 effective Cartier divisor.
 We define the \emph{proximity function} of $f$ with respect to $D$ by
 $$m_f(r,D)=\int_{\partial D(x)}\lambda_D\circ f(x)d\pi^r_o(x).$$
Now write $s_D=\tilde{s}_De$ locally, where $e$ is a local holomorphic frame of $(L_D,h).$
 The \emph{counting function}
 of $f$ with respect to $D$ is defined by
\begin{eqnarray*}
N_f(r,D)
&=& \pi \sum_{x\in f^*D\cap D(r)}g_r(o,x) \\
&=& \pi\int_{D(r)}g_r(o,x)dd^c\big{[}\log|\tilde{s}_D\circ f(x)|^2\big{]} \\
&=&\frac{1}{4}\int_{D(r)}g_r(o,x)\Delta_S\log|\tilde{s}_D\circ f(x)|^2dV(x).
\end{eqnarray*}
\noindent\textbf{Remark.} When $S=\mathbb C,$ the Green function is $(\log\frac{r}{|z|})/\pi$
 and the harmonic measure is $d\theta/2\pi.$ By integration by part,  we  observe that
it agrees with the classical ones.

 We introduce the concept of Nevanlinna constant proposed by Ru.
 \begin{defi}[\cite{ru0,ru1}]
 Let $L$
 be a holomorphic line bundle over  $X,$ and $D$ be an effective Cartier divisor  on $X.$
 If $X$ is  normal, then we define
 $${\rm{Nev}}(L,D)=\inf_{k,V,\mu}\frac{\dim_{\mathbb C}V}{\mu},$$
\noindent where $``\inf"$ is taken over all triples $(k,V,\mu)$ such that $V\subseteq H^0(X,kL)$ is a linear subspace
  with $\dim_{\mathbb C}V>1,$ and $\mu>0$ is a number with the property: for each $x\in {\rm{Supp}}D,$ there exists a basis $\mathscr B_x$ of $V$ such that
  $$\sum_{s\in\mathscr B_x}{\rm{ord}}_{E}(s)\geq\mu {\rm{ord}}_{E}(kD)$$
  for all irreducible components $E$ of $D$ passing through $x.$ If there exists  no such triples $(k,V,\mu),$ one defines ${\rm{Nev}}(L,D)=\infty.$
  If $X$ is not normal, then ${\rm{Nev}}(L,D)$ is defined by pulling back to the  normalization of  $X.$
\end{defi}

The main purpose of this paper is to explore the value distribution theory of holomorphic curves 
into complex projective varieties by extending 
 source space $\mathbb C$ to a general open Riemann surface. We 
prove  the following theorem
 \begin{theorem}\label{thm1} Let $L$ be a holomorphic line bundle
 over a normal complex projective variety $X$
with $\dim_{\mathbb C}H^0(X,kL)\geq1$ for some $k>0.$ Let $D$ be an effective Cartier divisor
on $X.$ Let $f:S\rightarrow X$ be a holomorphic curve with
Zariski-dense image. Then
$$m_f(r,D)
\leq {\rm{Nev}}(L,D)T_{f,L}(r)+o\big(T_{f,L}(r)\big)+O\Big(-\kappa(r)r^2+\log^+\log r\Big) \big\|,$$
where $\kappa$ is defined by $(\ref{kappa}),$ and  $``\|"$ means that the inequality holds for  $r>1$  outside a  subset of finite Lebesgue measure.
\end{theorem}
The  term $\kappa(r)r^2$ in the above theorem appears from the bending of metric of $S.$ In particular, when $S=\mathbb C,$ it deduces $\kappa(r)\equiv0$ and  $T_{f,L}(r)\geq O(\log r)$ for a holomorphic curve 
$f$ with Zariski-dense image in $X.$  As a consequence, we recover a result of Ru:
 \begin{cor}[\cite{ru0}]\label{} The same conditions are assumed as in Theorem $\ref{thm1}$. Let $f:\mathbb C\rightarrow X$ be a holomorphic curve with
Zariski-dense image. Then
$$m_f(r,D)
\leq {\rm{Nev}}(L,D)T_{f,L}(r)+o\big(T_{f,L}(r)\big)  \big\|.$$
\end{cor}
Theorem \ref{thm1} implies a defect relation
 \begin{cor}\label{}  The same conditions are assumed as in Theorem $\ref{thm1}$. Let $f:S\rightarrow X$ be a holomorphic curve with
Zariski-dense image satisfying  
$$\liminf_{r\rightarrow\infty}\frac{\kappa(r)r^2}{T_{f,L}(r)}=0.$$
Then
  $$\delta_f(D)\leq {\rm{Nev}}(L,D).$$
\end{cor}

\section{First Main Theorem}
\subsection{Stochastic formulas}~

We would use the stochastic method to study  value distribution theory for Riemann surfaces. To start with, we  introduce  Brownian motion and Dynkin formula \cite{13,itoo}.  It is known that the Dynkin formula plays a similar role as Green-Jensen formula \cite{ru}. Indeed,  the co-area formula is also introduced.

 Let  $(M,g)$ be a Riemannian manifold with  Laplace-Beltrami operator $\Delta_M$ associated to  $g.$  For $x\in M,$ we denote by $B_x(r)$ the geodesic ball centered at $x$ with radius $r,$ and denote by $S_x(r)$ the geodesic sphere centered at $x$ with radius $r.$
 By Sard's theorem, $S_x(r)$ is a submanifold of $M$ for almost every $r>0.$
A Brownian motion $X_t$ in $M$
is a heat diffusion  process  generated by $\frac{1}{2}\Delta_M$ with transition density function $p(t,x,y)$ which is the minimal positive fundamental solution of the  heat equation
  $$\frac{\partial}{\partial t}u(t,x)-\frac{1}{2}\Delta_{M}u(t,x)=0.$$
We denote by $\mathbb P_x$ the law of $X_t$ started at $x\in M$
 and by $\mathbb E_x$ the corresponding expectation with respect to $\mathbb P_x.$

  Let $D$  be a bounded domain with   smooth boundary $\partial D$ in $M$.
Fix $x\in D,$  we use $d\pi^{\partial D}_x$ to denote the harmonic measure  on $\partial D$ with respect to $x.$
This measure is a probability measure.
 Set
$$\tau_D:=\inf\big\{t>0:X_t\not\in D\big\}$$
which is a stopping time.
Denoted by $g_D(x,y)$  the Green function of $\Delta_M/2$ for  $D$ with a pole at $x$ and Dirichlet boundary condition, namely
$$-\frac{1}{2}\Delta_{M,y}g_D(x,y)=\delta_x(y), \ y\in D; \ \ g_D(x,y)=0, \ y\in \partial D,$$
where $\delta_x$ is the Dirac function.
 For $\phi\in \mathscr{C}_{\flat}(D)$
 (space of bounded continuous functions on $D$), the \emph{co-area formula} \cite{bass} asserts  that
$$ \mathbb{E}_x\left[\int_0^{\tau_D}\phi(X_t)dt\right]=\int_{D}g_{D}(x,y)\phi(y)dV(y).
$$
From Proposition 2.8 in \cite{bass}, we also have the relation of harmonic measures and hitting times that
\begin{equation}\label{hello123}
  \mathbb{E}_x\left[\psi(X_{\tau_{D}})\right]=\int_{\partial D}\psi(y)d\pi_x^{\partial D}(y)
\end{equation}
for any $\psi\in\mathscr{C}(\overline{D})$.

Let $u\in\mathscr{C}_\flat^2(M)$ (space of bounded $\mathscr{C}^2$-class functions on $M$), we have the famous \emph{It\^o formula} (see \cite{at,13, NN,itoo})
$$u(X_t)-u(x)=B\left(\int_0^t\|\nabla_Mu\|^2(X_s)ds\right)+\frac{1}{2}\int_0^t\Delta_Mu(X_s)dt, \ \ \mathbb P_x-a.s.$$
where $B_t$ is the standard  Brownian motion in $\mathbb R$ and $\nabla_M$ is gradient operator on $M$.
Take expectation of both sides of the above formula, it  follows \emph{Dynkin formula} (see \cite{at,itoo})
$$ \mathbb E_x[u(X_T)]-u(x)=\frac{1}{2}\mathbb E_x\left[\int_0^T\Delta_Mu(X_t)dt\right]
$$
for a stopping time $T$ such that each term  makes sense.

\noindent\textbf{Remark.} Thanks to expectation ``$\mathbb E_x$", the Dynkin formula, co-area formula and (\ref{hello123}) still work when $u, \phi$ or 
$\psi$ has a pluripolar set of 
singularities.
\subsection{First Main Theorem}~

Let $S$ be a complete open Riemann surface with  K\"ahler form $\alpha$ associated to  Hermitian metric $g.$
 Fix $o\in S,$ we
let $X_t$ be the Brownian
motion with generator $\Delta_S/2$ started at $o\in S.$ Moreover,  set a stopping time $$\tau_r=\inf\big\{t>0: X_t\not\in D(r)\big\}.$$
Let $$f:S\rightarrow X$$  be a holomorphic curve into a complex projective variety $X.$
Let $L\rightarrow X$
be an ample holomorphic line bundle equipped with Hermitian metric $h.$
 Apply co-area formula, we have
 $$T_{f,L}(r)=-\frac{1}{4}\mathbb E_o\left[\int_0^{\tau_r}\Delta_S\log h\circ f(X_t)dt\right].$$
 A relation of harmonic measures and hitting times implies that
 $$m_f(r,D)=\mathbb E_o\big[\lambda_D\circ f(X_{\tau_r})\big].$$

We here give the First Main Theorem of a holomorphic curve $f:S\rightarrow M$
 such that $f(o)\not\in {\rm{Supp}}D,$ where $D$ is an effective Cartier divisor on $X.$
Apply Dynkin formula to $\lambda_D\circ f(x),$
 $$\mathbb E_o\big[\lambda_D\circ f(X_{\tau_r})\big]-\lambda_D\circ f(o)=\frac{1}{2}
 \mathbb E_o\left[\int_0^{\tau_r}\Delta_S\lambda_D\circ f(X_t)dt\right].$$
The first term on the left hand side of the above equality is equal to $m_f(r,D),$ and the term on
the right hand side equals
$$\frac{1}{2}\mathbb E_o\left[\int_0^{\tau_r}\Delta_S\lambda_D\circ f(X_t)dt\right]=
\frac{1}{2}\int_{D(r)}g_r(o,x)\Delta_S\log\frac{1}{\|s_D\circ f(x)\|}dV(x)$$
due to co-area formula. Since $\|s_D\|^2=h|\tilde{s}_D|^2,$ where $h$ is a Hermitian metric on $L_D,$ we get
\begin{eqnarray*}
\frac{1}{2}\mathbb E_o\left[\int_0^{\tau_r}\Delta_S\lambda_D\circ f(X_t)dt\right] &=&
-\frac{1}{4}\int_{D(r)}g_r(o,x)\Delta_S\log h\circ f(x)dV(x) \\
&&-\frac{1}{4}\int_{D(r)}g_r(o,x)\Delta_S\log|\tilde{s}_D\circ f(x)|^2dV(x) \\
&=& T_{f,D}(r)-N_f(r,D).
\end{eqnarray*}
Therefore, we obtain
$$\text{F. M. T.}  \ \ \
  T_{f,D}(r)=m_f(r,D)+N_f(r,D)+O(1).$$
\textbf{Remark.}  $N_f(r,D)$ is of  a probabilistic expression 
$$N_f(r,D)=\lim_{\lambda\rightarrow\infty}\lambda\PP_o\left(\sup_{0\leq t\leq\tau_r}\log\frac{1}{\|s_D\circ f(X_t)\|}>\lambda\right).$$
\section{Logarithmic Derivative Lemma}

Let $(S,g)$ be a simply-connected and complete open Riemann surface with  Gauss curvature  $K_S\leq0$ associated to $g.$ By uniformization theorem,  there exists  a nowhere-vanishing holomorphic  vector field $\mathfrak X$ on $S.$
\subsection{Calculus Lemma}~

Let $\kappa$ be defined by (\ref{kappa}). As is noted before,
$\kappa$ is a non-positive, decreasing continuous function  on $[0,\infty).$
  Associate the ordinary differential equation
  \begin{equation}\label{G}
    G''(t)+\kappa (t)G(t)=0; \ \ \ G(0)=0, \ \ G'(0)=1.
  \end{equation}
 We compare (\ref{G})  with $y''(t)+\kappa(0)y(t)=0$ under the same  initial conditions,
 $G$ can be easily estimated  as
$$G(t)=t \ \ \text{for}  \ \kappa\equiv0; \ \ \ G(t)\geq t \ \ \text{for} \ \kappa\not\equiv0.$$
This implies that
\begin{equation}\label{vvvv}
  G(r)\geq r \ \ \text{for} \ r\geq0; \ \ \ \int_1^r\frac{dt}{G(t)}\leq\log r \ \ \text{for} \ r\geq1.
\end{equation}
On the other hand, we  rewrite (\ref{G}) as the form
$$\log'G(t)\cdot\log'G'(t)=-\kappa(t).$$
Since $G(t)\geq t$ is increasing,
then the decrease and non-positivity of $\kappa$ imply that for each fixed $t,$ $G$  must satisfy one of the following two inequalities
$$\log'G(t)\leq\sqrt{-\kappa(t)} \ \ \text{for} \ t>0; \ \ \ \log'G'(t)\leq\sqrt{-\kappa(t)} \ \ \text{for} \ t\geq0.$$
By virtue of $G(t)\rightarrow0$ as $t\rightarrow0,$ by integration, $G$ is bounded from above by
\begin{equation}\label{v2}
  G(r)\leq r\exp\big(r\sqrt{-\kappa(r)}\big) \ \  \text{for} \ r\geq0.
\end{equation}

 The main result of this subsection is the following 
\begin{theorem}[Calculus Lemma]\label{cal}
 Let $k\geq0$ be a locally integrable  function on $S$ such that it is locally bounded at $o\in S.$
 Then for any $\delta>0,$ there is a constant $C>0$ independent of $k,\delta,$ and a subset $E_\delta\subseteq(1,\infty)$ of finite Lebesgue measure such that
$$
\mathbb E_o\big{[}k(X_{\tau_r})\big{]}
\leq \frac{F(\hat{k},\kappa,\delta)e^{r\sqrt{-\kappa(r)}}\log r}{2\pi C}\mathbb E_o\left[\int_0^{\tau_{r}}k(X_{t})dt\right]
$$
  holds for $r>1$ outside $E_\delta,$  where  $\kappa$ is defined by $(\ref{kappa})$ and $F$ is defined by
$$
F\big{(}\hat{k},\kappa, \delta\big{)}
=\Big\{\log^+\hat{k}(r)\cdot\log^+\Big(re^{r\sqrt{-\kappa(r)}}\hat{k}(r)\big\{\log^{+}\hat{k}(r)\big\}^{1+\delta}\Big)\Big\}^{1+\delta} \ \ \ \
$$with
$$\hat k(r)=\frac{\log r}{C}\mathbb E_o\left[\int_0^{\tau_{r}}k(X_{t})dt\right].$$
Moreover, we have the estimate
$$\log F(\hat{k},\kappa,\delta)
\leq O\Big(\log^+\log \mathbb E_o\left[\int_0^{\tau_{r}}k(X_{t})dt\right]+\log^+r\sqrt{-\kappa(r)}+\log^+\log r\Big).
$$
\end{theorem}

To prove  theorem \ref{cal}, we need  some lemmas.
\begin{lemma}[\cite{atsuji}]\label{zz} Let $\eta>0$ be a constant. Then there is  a constant $C>0$ such that for
$r>\eta$ and $x\in B_o(r)\setminus \overline{B_o(\eta)}$
  $$g_r(o,x)\int_{\eta}^r\frac{dt}{G(t)}\geq C\int_{r(x)}^r\frac{dt}{G(t)}$$
 holds,  where  $G$ be defined by {\rm{(\ref{G})}}.
\end{lemma}

\begin{lemma}[\cite{ru}]\label{cal1} Let $T$ be a strictly positive nondecreasing  function of $\mathscr{C}^1$-class on $(0,\infty).$ Let $\gamma>0$ be a number such that $T(\gamma)\geq e,$ and $\phi$ be a strictly positive nondecreasing function such that
$$c_\phi=\int_e^\infty\frac{1}{t\phi(t)}dt<\infty.$$
Then, the inequality
  $T'(r)\leq T(r)\phi(T(r))$
holds for all $r\geq\gamma$ outside a subset of Lebesgue measure not exceeding $c_\phi.$ In particular, take $\phi(t)=\log^{1+\delta}t$ for a number  $\delta>0,$ then 
  $T'(r)\leq T(r)\log^{1+\delta}T(r)$
holds for all $r>0$ outside a subset $E_\delta\subseteq(0,\infty)$ of finite Lebesgue measure.
\end{lemma}
\emph{Proof of Theorem $\ref{cal}$}
\begin{proof}
The argument refers to Atsuji \cite{atsuji}. The simple-connectedness and the non-positivity of  Gauss curvature of $S$ imply the following relation (see \cite{Deb})
$$d\pi^r_{o}(x)\leq\frac{1}{2\pi r}d\sigma_r(x),$$
where  $d\sigma_r(x)$ is the induced volume measure on $\partial D(r).$
By  Lemma \ref{zz} and (\ref{vvvv}), we have
\begin{eqnarray*}
   \mathbb E_o\left[\int_0^{\tau_{r}}k(X_{t})dt\right] &=&
   \int_{D(r)}g_r(o,x)k(x)dV(x) \\ &=&\int_0^rdt\int_{\partial D(t)}g_r(o,x)k(x)d\sigma_t(x)  \\
&\geq& C\int_0^r\frac{\int_t^rG^{-1}(s)ds}{\int_1^rG^{-1}(s)ds}dt\int_{\partial D(t)}k(x)d\sigma_t(x) \\
&\geq& \frac{C}{\log r}\int_0^rdt\int_t^r\frac{ds}{G(s)}\int_{\partial D(t)}k(x)d\sigma_t(x), \\
 \mathbb E_o\big[k(X_{\tau_r})\big]&=&\int_{\partial D(r)}k(x)d\pi_o^r(x)\leq\frac{1}{2\pi r}\int_{\partial D(r)}k(x)d\sigma_r(x).
\end{eqnarray*}
Hence,
  \begin{eqnarray}\label{fr}
\mathbb E_o\left[\int_0^{\tau_{r}}k(X_{t})dt\right]&\geq&\frac{C}{\log r}\int_0^rdt\int_t^r\frac{ds}{G(s)}\int_{\partial D(o,t)}k(x)d\sigma_t(x), \nonumber \\
  \mathbb E_o\big[k(X_{\tau_r})\big]&\leq&\frac{1}{2\pi r}\int_{\partial D(r)}k(x)d\sigma_r(x).
\end{eqnarray}
 Set
$$\Lambda(r)=\int_0^rdt\int_t^r\frac{ds}{G(s)}\int_{\partial D(t)}k(x)d\sigma_t(x).$$
We conclude that
\begin{equation*}
 \Lambda(r)\leq\frac{\log r}{C}\mathbb E_o\left[\int_0^{\tau_{r}}k(X_{t})dt\right]=\hat{k}(r).
\end{equation*}
Since
$$\Lambda'(r)=\frac{1}{G(r)}\int_0^rdt\int_{\partial D(t)}k(x)d\sigma_t(x),$$
then it yields from (\ref{fr}) that
\begin{equation*}
  \mathbb E_o\big{[}k(X_{\tau_r})\big{]}\leq\frac{1}{2\pi r}\frac{d}{dr}\left(\Lambda'(r)G(r)\right).
\end{equation*}
Using Lemma \ref{cal1} twice with (\ref{v2}), then for any $\delta>0$
\begin{eqnarray*}
   && \frac{d}{dr}\left(\Lambda'(r)G(r)\right) \\
&\leq& G(r)\Big\{\log^+\Lambda(r)\cdot\log^+\left(G(r)\Lambda(r)\big\{\log^+\Lambda(r)\big\}^{1+\delta}\right)\Big\}^{1+\delta}\Lambda(r)  \\
&\leq& re^{r\sqrt{-\kappa(r)}}\Big\{\log^+\hat k(r)\cdot\log^+\Big(re^{r\sqrt{-\kappa(r)}}\hat k(r)\big\{\log^+\hat k(r)\big\}^{1+\delta}\Big)\Big\}^{1+\delta} \hat k(r) \\
&=& \frac{F\big{(}\hat k,\kappa,\delta\big{)}re^{r\sqrt{-\kappa(r)}}\log r}{C}\mathbb E_o\left[\int_0^{\tau_{r}}k(X_{t})dt\right] \ \ \
\end{eqnarray*}
holds outside a subset $E_\delta\subseteq(1,\infty)$ of finite Lebesgue measure.
Thus,
\begin{eqnarray*}
\mathbb E_o\big{[}k(X_{\tau_r})\big{]}
&\leq&\frac{F\big{(}\hat k,\kappa,\delta\big{)}e^{r\sqrt{-\kappa(r)}}\log r}{2\pi C}\mathbb E_o\left[\int_0^{\tau_{r}}k(X_{t})dt\right].
\end{eqnarray*}
Hence, we get the desired  inequality.
Indeed, for $r>1$  we compute that
$$
\log F(\hat k,\kappa,\delta)
\leq O\Big(\log^+\log^+\hat k(r)+\log^+r\sqrt{-\kappa(r)}+\log^+\log r\Big) \big\|$$
and
\begin{eqnarray*}
  \log^+\hat k(r)&\leq& \log \mathbb E_o\left[\int_0^{\tau_{r}}k(X_{t})dt\right]+\log^+\log r+O(1).
  \end{eqnarray*}
 We have arrived at the required estimate.
 \end{proof}
\subsection{Logarithmic Derivative Lemma}~

Let $\psi$ be a meromorphic function on $(S,g).$
The norm of the gradient of $\psi$ is defined by
$$\|\nabla_S\psi\|^2=\frac{1}{g}\left|\frac{\partial\psi}{\partial z}\right|^2$$
in a local coordinate $z.$ Locally, we write $\psi=\psi_1/\psi_0,$ where $\psi_0,\psi_1$ are local holomorphic functions without common zeros. Regard $\psi$  as a holomorphic mapping into $\mathbb P^1(\mathbb C)$  by
$x\mapsto[\psi_0(x):\psi_1(x)].$
We define
$$
T_\psi(r)=\frac{1}{4}\int_{D(r)}g_r(o,x)\Delta_S\log\big{(}|\psi_0(x)|^2+|\psi_1(x)|^2\big{)}dV(x)$$
and
$T(r,\psi):=m(r,\psi)+N(r,\psi)
$
with 
\begin{eqnarray*}
m(r,\psi)&=&\int_{\partial D(r)}\log^+|\psi(x)|d\pi^r_o(x), \\
N(r,\psi)&=&\pi \sum_{x\in \psi^{-1}(\infty)\cap D(r)}g_r(o,x).
\end{eqnarray*}
Let
  $i:\mathbb C\hookrightarrow\mathbb P^1(\mathbb C)$ be an inclusion defined by
 $z\mapsto[1:z].$  Via the pull-back by $i,$ we have a (1,1)-form $i^*\omega_{FS}=dd^c\log(1+|\zeta|^2)$ on $\mathbb C,$
 where $\zeta:=w_1/w_0$ and $[w_0:w_1]$ is
the homogeneous coordinate system of $\mathbb P^1(\mathbb C).$ The characteristic function of $\psi$ with respect to $i^*\omega_{FS}$ is defined by
$$\hat{T}_\psi(r) = \frac{1}{4}\int_{D(r)}g_r(o,x)\Delta_S\log(1+|\psi(x)|^2)dV(x).$$
Clearly, $\hat{T}_\psi(r)\leq T_\psi(r).$
We adopt the spherical distance $\|\cdot,\cdot\|$ on  $\mathbb P^1(\mathbb C),$ the proximity function of $\psi$  with respect to
$a\in \mathbb P^1(\mathbb C)$
is defined by
$$\hat{m}_\psi(r,a)=\int_{\partial D(r)}\log\frac{1}{\|\psi(x),a\|}d\pi_o^r(x).$$
Again,  set
$$\hat{N}_\psi(r,a)=\pi \sum_{x\in \psi^{-1}(a)\cap D(r)}g_r(o,x).$$
Then 
$\hat{T}_\psi(r)=\hat{m}_\psi(r,a)+\hat{N}_\psi(r,a)+O(1).$
 Note  that $m(r,\psi)=\hat{m}_\psi(r,\infty)+O(1),$ which  implies that
 $$
   T(r,\psi)=\hat{T}_\psi(r)+O(1), \ \ \ T\Big(r,\frac{1}{\psi-a}\Big)= T(r,\psi)+O(1).
$$
Hence, we arrive at 
\begin{equation}\label{relation}
  T(r,\psi)+O(1)=\hat{T}_\psi(r)\leq T_\psi(r)+O(1).
\end{equation}

We establish the following Logarithmic Derivative Lemma (LDL):
\begin{theorem}[LDL]\label{ldl2} Let $\psi$ be a nonconstant meromorphic function on $S.$ Let $\mathfrak{X}$ be a nowhere-vanishing holomorphic
 vector field on $S.$ Then
\begin{eqnarray*}
m\Big(r,\frac{\mathfrak{X}^k(\psi)}{\psi}\Big)  &\leq& \frac{3k}{2}\log T(r,\psi)+O\Big(\log^+\log T(r,\psi)-\kappa(r)r^2
+\log^+\log r\Big) \big{\|},
\end{eqnarray*}
where $\kappa$ is defined by $(\ref{kappa}).$
\end{theorem}
 On $\mathbb P^1(\mathbb C),$ we take a singular metric
$$\Phi=\frac{1}{|\zeta|^2(1+\log^2|\zeta|)}\frac{\sqrt{-1}}{4\pi^2}d\zeta\wedge d\overline \zeta.$$
A direct computation gives that
\begin{equation*}\label{}
\int_{\mathbb P^1(\mathbb C)}\Phi=1, \ \ \ 2\pi\psi^*\Phi=\frac{\|\nabla_S\psi\|^2}{|\psi|^2(1+\log^2|\psi|)}\alpha.
\end{equation*}
Set
\begin{equation*}\label{ffww}
  T_\psi(r,\Phi)=\frac{1}{2\pi}\int_{D(r)}g_r(o,x)\frac{\|\nabla_S\psi\|^2}{|\psi|^2(1+\log^2|\psi|)}(x)dV(x).
\end{equation*}
 By Fubini's theorem
\begin{eqnarray*}
T_\psi(r,\Phi)
&=&\int_{D(r)}g_r(o,x)\frac{\psi^*\Phi}{\alpha}dV(x)  \\
&=&\pi\int_{\zeta\in\mathbb P^1(\mathbb C)}\Phi\sum_{x\in \psi^{-1}(\zeta)\cap D(r)}g_r(o,x) \\
&=&\int_{\zeta\in\mathbb P^1(\mathbb C)}N_{\psi}(r,\zeta)\Phi
\leq T(r,\psi)+O(1).
\end{eqnarray*}
We get
\begin{equation}\label{get}
  T_\psi(r,\Phi)\leq T(r,\psi)+O(1).
\end{equation}
\begin{lemma}\label{999a} Assume that $\psi(x)\not\equiv0.$ Then 
\begin{eqnarray*}
  && \frac{1}{2}\mathbb E_o\left[\log^+\frac{\|\nabla_S\psi\|^2}{|\psi|^2(1+\log^2|\psi|)}(X_{\tau_r})\right] \\
  &\leq&\frac{1}{2}\log T(r,\psi)+O\big{(}\log^+\log T(r,\psi)+r\sqrt{-\kappa(r)}+\log^+\log r\big{)}  \big{\|},
\end{eqnarray*}
 where $\kappa$ is defined by $(\ref{kappa}).$
\end{lemma}
\begin{proof} By Jensen's inequality, it is clear that
\begin{eqnarray*}
   \mathbb E_o\left[\log^+\frac{\|\nabla_S\psi\|^2}{|\psi|^2(1+\log^2|\psi|)}(X_{\tau_r})\right]
   &\leq&  \mathbb E_o\left[\log\Big{(}1+\frac{\|\nabla_S\psi\|^2}{|\psi|^2(1+\log^2|\psi|)}(X_{\tau_r})\Big{)}\right] \nonumber \\
    &\leq& \log^+\mathbb E_o\left[\frac{\|\nabla_S\psi\|^2}{|\psi|^2(1+\log^2|\psi|)}(X_{\tau_r})\right]+O(1). \nonumber
\end{eqnarray*}
By Lemma \ref{cal} and (\ref{get})
\begin{eqnarray*}
   && \log^+\mathbb E_o\left[\frac{\|\nabla_S\psi\|^2}{|\psi|^2(1+\log^2|\psi|)}(X_{\tau_r})\right]  \\
   &\leq& \log^+\mathbb E_o\left[\int_0^{\tau_r}\frac{\|\nabla_M\psi\|^2}{|\psi|^2(1+\log^2|\psi|)}(X_{t})dt\right]
   +\log \frac{F(\hat{k},\kappa,\delta)e^{r\sqrt{-\kappa(r)}}\log r}{2\pi C}
    \\
   &\leq& \log T_\psi(r,\Phi)+\log F(\hat k,\kappa,\delta)+r\sqrt{-\kappa(r)}+\log^+\log r+O(1) \\
    &\leq& \log T(r,\psi)+O\Big(\log^+\log^+\hat k(r)+r\sqrt{-\kappa(r)}+\log^+\log r\Big) \big\|,
\end{eqnarray*}
where
$$\hat k(r)=\frac{\log r}{C}\mathbb E_o\left[\int_0^{\tau_{r}}\frac{\|\nabla_S\psi\|^2}{|\psi|^2(1+\log^2|\psi|)}(X_{t})dt\right].$$
Indeed, we note that
$$
\hat k(r)=\frac{2\pi \log r}{C}T_\psi(r,\Phi)\leq \frac{2\pi\log r}{C} T(r,\psi).
$$
Then we  have the desired inequality.
\end{proof}
We first prove LDL for the first-order derivative:
\begin{theorem}[LDL]\label{ldl1} Let $\psi$ be a nonconstant meromorphic function on $S.$ Let $\mathfrak{X}$ be a nowhere-vanishing holomorphic vector filed on $S.$ Then
\begin{eqnarray*}
m\Big(r,\frac{\mathfrak{X}(\psi)}{\psi}\Big)&\leq&\frac{3}{2}\log T(r,\psi)+O\Big{(}\log^+\log T(r,\psi)-\kappa(r)r^2+\log^+\log r\Big{)} \big{\|},
\end{eqnarray*}
where $\kappa$ is defined by $(\ref{kappa}).$
\end{theorem}
\begin{proof} Write $\mathfrak{X}=a\frac{\partial}{\partial z},$ then $\|\mathfrak{X}\|^2=g|a|^2.$  We have
\begin{eqnarray*}
 m\Big(r,\frac{\mathfrak{X}(\psi)}{\psi}\Big)
&=& \int_{\partial D(r)}\log^+\frac{|\mathfrak{X}(\psi)|}{|\psi|}(x)d\pi^r_o(x) \\
&\leq& \frac{1}{2}\int_{\partial D(r)}\log^+\frac{|\mathfrak{X}(\psi)|^2}{\|\mathfrak{X}\|^2|\psi|^2(1+\log^2|\psi|)}(x)d\pi^r_o(x) \\
&&+
\frac{1}{2}\int_{\partial D(r)}\log(1+\log^2|\psi(x)|)d\pi^r_o(x)+\frac{1}{2}\int_{\partial D(r)}\log^+\|\mathfrak{X}_x\|^2d\pi^r_o(x) \\
&:=& A+B+C.
\end{eqnarray*}
We next handle $A,
B,C$ respectively.  For $A,$ it yields from Lemma \ref{999a} that
\begin{eqnarray*}
A&=& \frac{1}{2}\int_{\partial D(r)}\log^+\frac{|a|^2\left|\frac{\partial \psi}{\partial z}\right|^2}{g|a|^2|\psi|^2(1+\log^2|\psi|)}(x)d\pi^r_o(x) \\
&=& \frac{1}{2}\int_{\partial D(r)}\log^+\frac{\|\nabla_S\psi\|^2}{|\psi|^2(1+\log^2|\psi|)}(x)d\pi^r_o(x) \\
 &\leq&\frac{1}{2}\log T(r,\psi)+O\Big{(}\log^+\log T(r,\psi)+r\sqrt{-\kappa(r)}+\log^+\log r\Big{)} \big{\|}.
\end{eqnarray*}
For $B,$ the Jensen's inequality implies that
\begin{eqnarray*}
B &\leq&  \int_{\partial D(r)}\log\Big(1+\log^+|\psi(x)|+\log^+\frac{1}{|\psi(x)|}\Big)d\pi^r_o(x) \\
&\leq& \log\int_{\partial D(r)}\Big(1+\log^+|\psi(x)|+\log^+\frac{1}{|\psi(x)|}\Big)d\pi^r_o(x) \\
&\leq& \log T(r,\psi)+O(1).
\end{eqnarray*}
Finally, we estimate $C.$ By the condition,  $\|\mathfrak{X}\|>0.$ Since $S$ is non-positively curved and $a$ is holomorphic, then $\log\|\mathfrak{X}\|$ is subharmonic, i.e., $\Delta_S\log\|\mathfrak X\|\geq0.$
Clearly, we have
$$\Delta_S\log^+\|\mathfrak{X}\|\leq \Delta_S\log \|\mathfrak{X}\|$$
for $x\in S$ satisfying $\|\mathfrak{X}_x\|\neq1.$ Notice that  
$\log^+\|\mathfrak{X}_x\|=0$
for $x\in S$ satisfying $\|\mathfrak{X}_x\|\leq1.$ 
Thus, by Dynkin formula we have
\begin{eqnarray}\label{ok}
\ \ \ \ \ C &=& \frac{1}{2}\mathbb E_o\left[\log^+\|\mathfrak{X}(X_{\tau_r})\|^2\right] \\
&\leq&  \frac{1}{4}\mathbb E_o\left[\int_0^{\tau_r}\Delta_S\log\|\mathfrak{X}(X_t)\|^2dt \right]+O(1) \nonumber \\
 &=& \frac{1}{4}\mathbb E_o\left[\int_0^{\tau_r}\Delta_S\log g(X_t)dt \right]+
 \frac{1}{4}\mathbb E_o\left[\int_0^{\tau_r}\Delta_S\log|a(X_t)|^2dt \right]+O(1) \nonumber \\
 &=& -\mathbb E_o\left[\int_0^{\tau_r}K_S(X_t)dt \right]+O(1) \nonumber \\
 &\leq& -\kappa(r)\mathbb E_o\big[\tau_r\big]+O(1), \nonumber
\end{eqnarray}
where we use the fact $K_S=-(\Delta_S\log g)/4.$ Thus, we prove the theorem by using  $\mathbb E_o\left[\tau_r\right]\leq r^2/2$ which is due
to Lemma \ref{yyyy} below.
\end{proof}
 \begin{lemma}\label{yyyy} Let $X_t$ be a Brownian motion in a simply-connected complete Riemann surface $S$ of non-positive Gauss curvature. Then
$$\mathbb E_o\big[\tau_r\big]\leq\frac{r^2}{2}.$$
\end{lemma}
\begin{proof}  We refer to arguments of  Atsuji \cite{atsuji}. Apply It$\rm{\hat{o}}$ formula to $r(x)$
\begin{equation}\label{kiss}
  r(X_t)=B_t-L_t+\frac{1}{2}\int_0^t\Delta_Sr(X_s)ds,
\end{equation}
where $B_t$ is the standard Brownian motion in $\mathbb R,$  $L_t$ is the local time on cut locus of $o,$ an increasing process
which increases only at cut loci of $o.$ Since $S$ is simply connected and  non-positively  curved, then
$$\Delta_Sr(x)\geq\frac{1}{r(x)},\ \ L_t\equiv0.$$
By (\ref{kiss}), we arrive at 
$$r(X_t)\geq B_t+\frac{1}{2}\int_0^t\frac{ds}{r(X_s)}.$$
Associate the stochastic differential equation
$$ dW_t=dB_t+\frac{1}{2}\frac{dt}{W_t}, \ \ W_0=0,$$
where $B_t$ is the  standard Brownian motion in $\mathbb R,$ and
 $W_t$ is the $2$-dimensional Bessel process  defined as the Euclidean norm of Brownian motion in
$\mathbb R^2.$
By the standard comparison arguments of stochastic differential equations, one gets that
\begin{equation}\label{yss}
  W_t\leq r(X_t)
\end{equation}
almost surely. Set
$$\iota_r=\inf\big{\{}t>0:W_t\geq r\big{\}},$$
which is a stopping time. From (\ref{yss}), we can verify that
$ \iota_r\geq\tau_r.$
This implies
\begin{equation}\label{ssy}
  \mathbb E_o[\iota_r]\geq\mathbb E_o[\tau_r].
\end{equation}
Since $W_t$ is the Euclidean norm of the Brownian motion in $\mathbb R^{2}$ starting from the origin, then applying Dynkin formula to $W_t^2$ we have
$$\mathbb E_o[W_{\iota_r}^2]=\frac{1}{2}\mathbb E_o\left[\int_0^{\iota_r}\Delta_{\mathbb R^{2}} W_t^2dt\right]=2\mathbb E_o[\iota_r],$$
where $\Delta_{\mathbb R^{2}}$ is the Laplace operator  on $\mathbb R^{2}.$
Using (\ref{yss}) and (\ref{ssy}), we conclude that
$$r^2=\mathbb E_o[r^2]=2\mathbb E_o[\iota_r]\geq2\mathbb E_o[\tau_r].$$
This certifies the assertion.
\end{proof}
\emph{Proof of Theorem $\ref{ldl2}$}
\begin{proof} Note that
\begin{eqnarray*}
m\Big(r,\frac{\mathfrak{X}^k(\psi)}{\psi}\Big)&\leq& \sum_{j=1}^k m\Big(r,\frac{\mathfrak{X}^j(\psi)}{\mathfrak{X}^{j-1}(\psi)}\Big).
\end{eqnarray*}
Therefore, we finish the proof by using  Lemma \ref{hello1} below.
\end{proof}
\begin{lemma}\label{hello1} We have
\begin{eqnarray*}
m\Big(r,\frac{\mathfrak{X}^{k+1}(\psi)}{\mathfrak{X}^{k}(\psi)}\Big)
&\leq& \frac{3}{2}\log T(r,\psi)+O\Big(\log^+\log T(r,\psi)-\kappa(r)r^2+\log^+\log r\Big) \big{\|},
\end{eqnarray*}
 where $\kappa$ is defined by $(\ref{kappa}).$
\end{lemma}
\begin{proof}
We first claim that
\begin{eqnarray}\label{hello}
\ \ T\big{(}r,\mathfrak{X}^k(\psi)\big{)}
&\leq& 2^kT(r,\psi)+O\Big(\log T(r,\psi)-\kappa(r)r^2+\log^+\log r\Big) \big\|.
\end{eqnarray}
By virtue of Theorem \ref{ldl1}, when $k=1$
\begin{eqnarray*}
T(r,\mathfrak{X}(\psi))
&=& m(r,\mathfrak{X}(\psi))+N(r,\mathfrak{X}(\psi)) \\
&\leq& m(r,\psi)+2N(r,\psi)+m\Big(r,\frac{\mathfrak{X}(\psi)}{\psi}\Big) \\
&\leq& 2T(r,\psi)+m\Big(r,\frac{\mathfrak{X}(\psi)}{\psi}\Big) \\
&\leq& 2T(r,\psi)+O\Big(\log T(r,\psi)-\kappa(r)r^2+\log^+\log r\Big) \big\|
\end{eqnarray*}
holds for $r>1$ outside a set of finite Lebesgue measure.
Assuming now that the claim holds for $k\leq n-1.$ By induction, we only need to prove the claim in the case when $k=n.$ By the claim for $k=1$ proved above and Theorem \ref{ldl1} repeatedly, we have
\begin{eqnarray*}
 T\big{(}r,\mathfrak{X}^n(\psi)\big{)}
&\leq& 2T\big{(}r,\mathfrak{X}^{n-1}(\psi)\big{)}+O\Big(\log T\big{(}r,\mathfrak{X}^{n-1}(\psi)\big{)}-\kappa(r)r^2+\log^+\log r\Big) \\
&\leq& 2^{n}T(r,\psi)+O\Big(\log T(r,\psi)-\kappa(r)r^2+\log^+\log r\Big) \\
&&
+O\Big(\log T\big{(}r,\mathfrak{X}^{n-1}(\psi)\big{)}-\kappa(r)r^2+\log^+\log r\Big) \\
&\leq& 2^nT(r,\psi)+O\Big(\log T(r,\psi)-\kappa(r)r^2+\log^+\log r\Big) \\
&&+O\left(\log T\big{(}r,\mathfrak{X}^{n-1}(\psi)\big{)}\right) \\
&& \cdots\cdots\cdots \\
&\leq& 2^nT(r,\psi)+O\Big(\log T(r,\psi)-\kappa(r)r^2+\log^+\log r\Big) \big\|. \ \ \ \ \
\end{eqnarray*}
The claim (\ref{hello}) is proved.
 Using Theorem \ref{ldl1} and (\ref{hello}) to get
\begin{eqnarray*}
&& m\left(r,\frac{\mathfrak{X}^{k+1}(\psi)}{\mathfrak{X}^{k}(\psi)}\right) \\
&\leq& \frac{3}{2}\log T\big{(}r,\mathfrak{X}^k(\psi)\big{)}+O\Big(\log^+\log T(r,\mathfrak{X}^k(\psi))-\kappa(r)r^2+\log^+\log r\Big) \\
&\leq& \frac{3}{2}\log T(r,\psi)+O\Big(\log^+\log T(r,\psi)-\kappa(r)r^2+\log^+\log r\Big) \big\|. \ \ \ \ \
\end{eqnarray*}
\end{proof}

\section{Second Main Theorem}
\subsection{Wronskian determinants}~

Let $S$ be an open Riemann surface with  a nowhere-vanishing holomorphic vector field $\mathfrak{X}$ (it always exists), 
which is equipped with a complete Hermitian metric $h$ such that the  Gauss curvature $K_S\leq0.$
Let $$f:S\rightarrow\mathbb P^n(\mathbb C)$$ be a holomorphic curve into complex projective space with the Fubini-Study form $\omega_{FS}.$
 Locally, we may write $f=[f_0:\cdots:f_n],$  a reduced representation, i.e., $f_0=w_0\circ f,\cdots$ are local holomorphic functions without common zeros,
  where $w=[w_0:\cdots:w_n]$ denotes homogenous coordinate system of $\mathbb P^n(\mathbb C).$ Set $\|f\|^2=|f_0|^2+\cdots+|f_n|^2.$
  Notice that $\Delta_S\log\|f\|^2$ is independent of the choices of representations of $f,$ so it is well  defined on $S.$
 The height function of $f$ is defined by
$$
 T_f(r)= \pi\int_{D(r)}g_r(o,x)f^*\omega_{FS} 
 =\frac{1}{4}\int_{D(r)}g_r(o,x)\Delta_S\log\|f(x)\|^2dV(x). 
$$
 Let $H$ be a hyperplane  of $\mathbb P^n(\mathbb C)$ with defining function
$\hat{H}(w)=a_0w_0+\cdots+a_nw_n.$ Set $\|\hat{H}\|^2=|a_0|^2+\cdots+|a_n|^2.$
The counting function of $f$ with respect to $H$ is defined by
\begin{eqnarray*}
N_f(r,H) &=& \pi\int_{D(r)}g_r(o,x)dd^c\big{[}\log|\hat{H}\circ f|^2\big{]} \\
&=&\frac{1}{4}\int_{D(r)}g_r(o,x)\Delta_S\log|\hat{H}\circ f|^2dV(x).
\end{eqnarray*}
We define the proximity function of $f$ with respect to $H$  by
$$m_f(r,H)=\int_{\partial D(r)}\log\frac{\|\hat{H}\|\|f(x)\|}{|\hat{H}\circ f(x)|}d\pi_o^r(x).$$
\begin{lemma}\label{v000} Assume that $f_k\not\equiv0$ for some $k.$ We have
$$\max_{0\leq j\leq n}T\Big(r,\frac{f_j}{f_k}\Big)\leq T_f(r)+O(1).$$
\end{lemma}
\begin{proof} By (\ref{relation}), we  arrive at
\begin{eqnarray*}
 T\Big(r,\frac{f_j}{f_k}\Big)
 &\leq&T_{f_j/f_k}(r)+O(1)     \\
 &\leq& \frac{1}{4}\int_{D(r)}g_r(o,x)\Delta_S\log\Big{(}\sum_{j=0}^n|f_j(x)|^2\Big{)}dV(x)+O(1) \\
  &=& T_{f}(r)+O(1).
\end{eqnarray*}
\end{proof}

 Let
 $H_1,\cdots, H_q$ be $q$ hyperplanes of $\mathbb P^n(\mathbb C)$ in $N$-subgeneral position with defining functions given by
$$\hat{H}_j(w)=\sum_{k=0}^na_{jk}w_{k}, \ \ 1\leq j\leq q.$$
 One  defines Wronskian determinant and logarithmic
Wronskian determinant of $f$ with respect to $\mathfrak X$ respectively  by
$$W_\mathfrak{X}(f_0,\cdots,f_n)=\left|
  \begin{array}{ccc}
   f_0 & \cdots & f_n \\
   \mathfrak{X}(f_0) & \cdots & \mathfrak{X}(f_n) \\
    \vdots & \vdots & \vdots \\
     \mathfrak{X}^n(f_0) & \cdots & \mathfrak{X}^n(f_n) \\
  \end{array}
\right|, \ \ \Delta_\mathfrak{X}(f_0,\cdots,f_n)=\left|
  \begin{array}{ccc}
   1 & \cdots & 1 \\
   \frac{\mathfrak{X}(f_0)}{f_0} & \cdots & \frac{\mathfrak{X}(f_n)}{f_n} \\
    \vdots & \vdots & \vdots \\
     \frac{\mathfrak{X}^n(f_0)}{f_0} & \cdots & \frac{\mathfrak{X}^n(f_n)}{f_n} \\
  \end{array}
\right|.$$
For a $(n+1)\times(n+1)$-matrix $A$  and a nonzero meromorphic function $\phi$  on $S,$ we can check the following basic properties:
\begin{eqnarray*}
\Delta_\mathfrak{X}(\phi f_0,\cdots,\phi f_n)&=&\Delta_\mathfrak{X}(f_0,\cdots,f_n), \\
W_\mathfrak{X}(\phi f_0,\cdots,\phi f_n)&=&\phi^{n+1}W_\mathfrak{X}(f_0,\cdots,f_n), \\
W_\mathfrak{X}\big{(}(f_0,\cdots,f_n)A\big{)}&=&\det(A)W_\mathfrak{X}(f_0,\cdots,f_n), \\
W_\mathfrak{X}(f_0,\cdots,f_n)&=&\Big{(}\prod_{j=0}^nf_j\Big{)}\Delta_\mathfrak{X}(f_0,\cdots,f_n).
\end{eqnarray*}
Obviously,  $\Delta_\mathfrak{X}(f_0,\cdots,f_n)$ is globally well defined on $S.$
 \begin{lemma}\label{ldl3} Let $Q\subseteq\{1,\cdots,q\}$ with $|Q|=n+1.$ If $S$ is simply connected, then we have
 \begin{eqnarray*}
 m\Big(r,\Delta_\mathfrak{X}\big{(}\hat{H}_k\circ f, k\in Q\big{)}\Big)
 &\leq& O\Big(\log T_f(r)-\kappa(r)r^2+\log^+\log r\Big) \big{\|},
 \end{eqnarray*}
where $\kappa$ is defined by $(\ref{kappa}).$
 \end{lemma}
 \begin{proof} We write $Q=\{j_0,\cdots,j_n\}$ and suppose that $\hat{H}_{j_0}\circ f\not\equiv0$ without loss of generality.
 The property of logarithmic Wronskian determinant implies
 $$\Delta_\mathfrak{X}\big{(}\hat{H}_{j_0}\circ f,\cdots, \hat{H}_{j_n}\circ f\big{)}=
 \Delta_\mathfrak{X}\left(1,\frac{\hat{H}_{j_1}\circ f}{\hat{H}_{j_0}\circ f}, \cdots, \frac{\hat{H}_{j_n}\circ f}{\hat{H}_{j_0}\circ f}\right).$$
 Since $\hat{H}_{j_0}\circ f,\cdots,\hat{H}_{j_n}\circ f$ are linear forms of $f_0,\cdots, f_n,$ by Theorem \ref{ldl2} and
 Lemma
 \ref{v000} we have
 $$ m\Big(r,\Delta_\mathfrak{X}\big{(}\hat{H}_k\circ f, k\in Q\big{)}\Big)  \\
  \leq O\Big(\log T_f(r)-\kappa(r)r^2+\log^+\log r\Big) \big\|.$$
 We have arrived at the desired inequality.
 \end{proof}
\begin{lemma}\label{yyy}  Let $H_1\cdots,H_q$ be hyperplanes of $\mathbb P^n(\mathbb C).$
 Let $f:S\rightarrow \mathbb P^n(\mathbb C)$ be a linearly non-degenerate holomorphic curve. Then
\begin{eqnarray*}
&& \int _{\partial D(r)}\max_{Q}\sum_{k\in Q}\log\frac{\|\hat{H}_k\|\|f(x)\|}{|\hat{H}_k\circ f(x)|}d\pi^r_o(x) \\
&\leq&  (n+1)T_f(r)+O\Big(\log T_f(r)-\kappa(r)r^2+\log^+\log r\Big)  \big\|,
\end{eqnarray*}
where $Q$ ranges over all subsets of $\{1,\cdots,q\}$ such that $\{\hat{H}_k\}_{k\in Q}$ are linearly independent.
\end{lemma}
\begin{proof} Without loss of generality, we assume  that $q\geq n+1$ and $H_1\cdots,H_q$ are in general position. Then
\begin{eqnarray*}
   && \int _{\partial D(r)}\max_{Q}\sum_{k\in Q}\log\frac{\|\hat{H}_k\|\|f(x)\|}{|\hat{H}_k\circ f(x)|}d\pi^r_o(x) \\
   &=& \int _{\partial D(r)}\max_{|Q|=n+1}\log\prod_{k\in Q}\frac{\|\hat{H}_k\|\|f(x)\|}{|\hat{H}_k\circ f(x)|}d\pi^r_o(x) \\
   &\leq& \int _{\partial D(r)}\max_{|Q|=n+1}\log\frac{\|f(x)\|^{n+1}}
   {\prod_{k\in Q}|\hat{H}_k\circ f(x)|}d\pi^r_o(x)+O(1) \\
   &=& \int _{\partial D(r)}\max_{|Q|=n+1}
   \log \frac{\big|\Delta_\mathfrak{X}\big(\hat{H}_k\circ f(x),\ k\in Q\big)\big|\|f(x)\|^{n+1}}
   {\big|W_\mathfrak{X}\big(\hat{H}_k\circ f(x),\ k\in Q\big)\big|}d\pi^r_o(x)+O(1).
\end{eqnarray*}
By $W_\mathfrak{X}(\hat{H}_k\circ f,\ k\in Q)=b_QW_\mathfrak{X}(f_0,\cdots,f_n)$ (with $|Q|=n+1$) for a nonzero constant $b_Q$ depending on $Q,$ we further conclude that
\begin{eqnarray*}
   && \int _{\partial D(r)}\max_{Q}\sum_{k\in Q}\log\frac{\|\hat{H}_k\|\|f(x)\|}{|\hat{H}_k\circ f(x)|}d\pi^r_o(x) \\
  &\leq&
   \int _{\partial D(r)}\max_{|Q|=n+1}\log \Big|\Delta_\mathfrak{X}\big(\hat{H}_k\circ f(x),\ k\in Q\big)\Big|d\pi^r_o(x) \\
   &&+\int _{\partial D(r)} \log\frac{\|f(x)\|^{n+1}}{\big|W_\mathfrak{X}\big(f_0(x),\cdots,f_n(x)\big)\big|}d\pi^r_o(x)+O(1) \\
   &:=& A+B+O(1).
\end{eqnarray*}
We next handle the terms $A$ and $B.$ By Lemma \ref{ldl3}
\begin{eqnarray*}
  A &\leq& \int _{\partial D(r)}\log \sum_{|Q|=n+1}\Big|\Delta_\mathfrak{X}\big(\hat{H}_k\circ f(x),\ k\in Q\big)\Big|d\pi^r_o(x) \\
  &\leq& \sum_{|Q|=n+1}m\Big(r,\Delta_\mathfrak{X}\big{(}\hat{H}_k\circ f, k\in Q\big{)}\Big)
  +O(1) \\
  &\leq& O\Big(\log T_f(r)-\kappa(r)r^2+\log^+\log r\Big)  \big\|.
\end{eqnarray*}
Apply Dynkin formula to $B,$
\begin{eqnarray*}
  B &=& \frac{1}{2}\int_{D(r)}g_r(o,x)\Delta_S\log\frac{\|f(x)\|^{n+1}}{\big|W_\mathfrak{X}\big(f_0(x),\cdots,f_n(x)\big)\big|}dV(x)+O(1) \\
  &=& (n+1)T_f(r)-N_{W_\mathfrak{X}}(r,0)+O(1) \\
   &\leq& (n+1)T_f(r)+O(1).
\end{eqnarray*}
Putting together the above, we have the desired inequality.
\end{proof}
\begin{theorem}\label{xy12} Let $D$ be an effective Cartier divisor on a complex projective variety $M.$ Let $s_1,\cdots,s_q$ be nonzero
 elements of a nonzero linear subspace  $V\subseteq H^0(M, L_D).$ Let $f:S\rightarrow M$ be a holomorphic curve with Zariski dense image. Then
\begin{eqnarray*}
&& \int _{\partial D(r)}\max_{Q}\sum_{k\in Q}\lambda_{s_k}\circ f(x)d\pi^r_o(x) \\
&\leq&  \dim_{\mathbb C}VT_{f,D}(r)+O\Big(\log T_{f,D}(r)-\kappa(r)r^2+\log^+\log r\Big)  \big\|,
\end{eqnarray*}
where  $\lambda_{s_k}$ denotes the Weil function of $(s_k),$ and $Q$ ranges over all subsets of $\{1,\cdots,q\}$ such that $\{s_k\}_{k\in Q}$ are linearly independent.
\end{theorem}
\begin{proof} Set $d=\dim_{\mathbb C}V.$ If $d=1,$ then $|Q|=1.$ Hence, for  $1\leq j\leq q,$ we have $s_k=b_ks_D$  for some constant $b_k\neq0.$
By the First Main Theorem
\begin{eqnarray*}
\int _{\partial D(r)}\max_{Q}\sum_{k\in Q}\lambda_{s_k}\circ f(x)d\pi^r_o(x)
&\leq& \int _{\partial D(r)}\lambda_{D}\circ f(x)d\pi^r_o(x)+O(1) \\
&\leq&  T_{f,D}(r)+O(1).
\end{eqnarray*}
\ \ \ \ If $d>1,$  we treat the projective space $P(V)$ of $V$ that can be regarded as $\mathbb P^{d-1}(\mathbb C).$  Let $M'$ be the closure of  graph of $f,$ then there exists the canonical projection morphisms
 $\pi:M'\rightarrow M$ and $\phi:M'\rightarrow \mathbb P^{d-1}(\mathbb C).$ Now we lift $f$ to $\tilde{f}:S\rightarrow M'.$ Notice that (see \cite{ru1}) there is an effective Cartier 
divisor $B$ on $M'$ such that for each $s\in V,$ we can choose a hyperplane $H_s$ (depending on $s$) of $\mathbb P^{d-1}(\mathbb C)$ which satisfies
$\pi^*(s)-B=\phi^*H_s$ (more precisely, $\phi^*\mathscr O(1)\cong L_{\pi^*D-B}$). For $1\leq j\leq q,$ one chooses hyperplanes $H_j$ of $\mathbb P^{d-1}(\mathbb C)$ such that
$\pi^*(s_j)-B=\phi^*H_j.$ Since $M$ is compact,  then we have
\begin{equation}\label{xxx}
  \lambda_{\pi^*(s_j)}=\lambda_{\phi^*H_j}+\lambda_{B}+O(1).
\end{equation}
In further, we have
\begin{eqnarray}
\label{aa1} N_{\tilde{f}}\big(r,\pi^*(s_j)\big)&=& N_{\tilde{f}}\big(r,\phi^*H_{j}\big)+N_{\tilde{f}}(r,B), \\
\label{aa2} m_{\tilde{f}}\big(r,\pi^*(s_j)\big)&=& m_{\tilde{f}}\big(r,\phi^*H_{j}\big)+m_{\tilde{f}}(r,B)+O(1).
\end{eqnarray}
Note that $\phi\circ\tilde{f}:S\rightarrow\mathbb P^{d-1}(\mathbb C)$ is a holomorphic curve,  using the First Main Theorem, it yields that
 $T_{\phi\circ\tilde{f}}(r)=m_{\phi\circ\tilde{f}}(r,H_{j})+N_{\phi\circ\tilde{f}}(r,H_{j})+O(1).$
  Indeed,  $L_{(s_j)}\cong L_D$ and $f=\pi\circ \tilde{f}$ are noted. By (\ref{aa1}) and (\ref{aa2}),  we arrive at
\begin{equation}\label{f11}
  T_{f,D}(r) = T_{\phi\circ\tilde{f}}(r)+T_{\tilde{f},B}(r)+O(1).
\end{equation}
Combining (\ref{xxx}) with $\lambda_{s_{j}}\circ f=\lambda_{\pi^*(s_j)}\circ \tilde{f}+O(1),$ it suffices  to show that
\begin{eqnarray*}
&& \int _{\partial D(r)}\max_{Q}\sum_{k\in Q}\Big(
\lambda_{H_k}\circ \phi\circ \tilde{f}(x)+\lambda_B\circ \tilde{f}(x)\Big)d\pi^r_o(x) \\
&\leq&  dT_{f,D}(r)+O\Big(\log T_{f,D}(r)-\kappa(r)r^2+\log^+\log r\Big)  \big\|.
\end{eqnarray*}
In fact, by Lemma \ref{yyy} and (\ref{f11}) we have
\begin{eqnarray*}
&& \int _{\partial D(r)}\max_{Q}\sum_{k\in Q}
\lambda_{H_k}\circ \phi\circ \tilde{f}(x)d\pi^r_o(x)
 \\
&=& \int _{\partial D(r)}\max_{Q}\sum_{k\in Q}\log\frac{\|\hat{H}_k\|\|\phi\circ \tilde{f}(x)\|}
{|\hat{H}_k\circ\phi\circ \tilde{f}(x)|}
d\pi^r_o(x)+O(1)
 \\
 &\leq&  dT_{\phi\circ \tilde{f}}(r)+O\Big(\log T_{\phi\circ \tilde{f}}(r)-\kappa(r)r^2+\log^+\log r\Big)  \\
&\leq&  d\big(T_{f,D}(r)-T_{\tilde{f},B}(r)\big)+O\Big(\log T_{f,D}(r)-\kappa(r)r^2+\log^+\log r\Big) \big\|.
\end{eqnarray*}
 Since $|Q|\leq d,$
  the First Main Theorem implies that
$$\int _{\partial D(r)}\max_{Q}\sum_{k\in Q}\lambda_B\circ \tilde{f}(x)d\pi^r_o(x)\leq dT_{\tilde{f},B}(r)+O(1).$$
Combining the above, we conclude the proof.
\end{proof}

\subsection{Second Main Theorem}~

In this subsection, we aim to prove the main theorem of the paper, namely, the Second Main Theorem (Theorem \ref{thm1}).

Let $S$ be a complete open Riemann surface with  Gauss curvature $K_S\leq0.$
 We here consider the universal covering $\pi:\tilde{S}\rightarrow S.$ By the pull-back of $\pi,$ $\tilde{S}$ could be equipped with the induced metric
 from the
metric of $S.$ In such case, $\tilde{S}$ is a simply-connected and complete open Riemann surface of non-positive Gauss curvature.
Take a diffusion process $\tilde{X}_t$ in $\tilde{S}$  so that $X_t=\pi(\tilde{X}_t),$ then
$\tilde{X}_t$ becomes a Brownian motion with generator  $\Delta_{\tilde{S}}/2$ which is induced from the pull-back metric.
Let $\tilde{X}_t$ start from $\tilde{o}\in\tilde{S}$  with $o=\pi(\tilde{o}),$ then we have
$$\mathbb E_o[\phi(X_t)]=\mathbb E_{\tilde{o}}\big{[}\phi\circ\pi(\tilde{X}_t)\big{]}$$
for $\phi\in \mathscr{C}_{\flat}(S).$ Set $$\tilde{\tau}_r=\inf\big{\{}t>0: \tilde{X}_t\not\in \tilde{D}(r)\big{\}},$$ where
$\tilde{D}(r)$ is a geodesic disc centered at $\tilde{o}$ with radius $r$ in $\tilde{S}.$
 If necessary, one can extend the filtration in probability space where $(X_t,\mathbb P_o)$ are defined so that $\tilde{\tau}_r$ is a stopping time with
 respect to a filtration where the stochastic calculus of $X_t$ works.
By the above arguments, we would assume  $S$  is simply connected by lifting $f$ to the covering.

\emph{Proof of Theorem $\ref{thm1}$}
\begin{proof}
Let $\mathfrak{P}$ be the set of all prime divisors occurring in $D,$ then
$$D=\sum_{E\in\mathfrak{P}}{\rm{ord}}_E(D)\cdot E.$$
Set
$\Lambda=\{\sigma\subseteq\mathfrak{P}: \cap_{E\in\sigma}E\neq\emptyset\}$ which is a finite set.
For any $\sigma\in\Lambda,$ we write $D=D_{\sigma,1}+D_{\sigma,2},$ where
$$D_{\sigma,1}=\sum_{E\in\sigma}{\rm{ord}}_E(D)\cdot E, \ \ \ D_{\sigma,2}=\sum_{E\not\in\sigma}{\rm{ord}}_E(D)\cdot E.$$
From the definition of ${\rm{Nev}}(L,D),$ for each $\sigma\in\Lambda,$ there exists a basis $\mathscr B_{\sigma}$
of a linear subspace
$V_k\subseteq H^0(X,kL)$   with $\dim_\mathbb CV_k>1$ (for some $k$) such that
$$\frac{1}{\mu_k}\sum_{s\in\mathscr B_{\sigma}}{\rm{ord}}_E(s)\geq{\rm{ord}}_E(kD)$$
 at some (and hence all) points $x\in \cap_{E\in\sigma}E.$ For each $E\in\sigma,$ we have
 \begin{equation}\label{ww2}
  \frac{1}{\mu_k}\sum_{s\in\mathscr B_{\sigma}}{\rm{ord}}_E(s)\cdot\lambda_E\geq{\rm{ord}}_E(kD)\cdot\lambda_E.
\end{equation}
Note that (refer to the proof of Proposition 3.1 in \cite{ru0}) there exists  a number $B>0$ such that for each $x\in X,$
one can pick  $\sigma_x\in\Lambda$ (depending on $x$) such that
$\lambda_{D_{\sigma_{x},2}}(x)\leq B,$ here $B$ is independent of $x$. Thus,
\begin{equation}\label{ww1}
  \lambda_D(x)\leq\lambda_{D_{\sigma_{x},1}}(x)+O(1).
\end{equation}
 By properties of Weil functions,
we have from (\ref{ww1}) and (\ref{ww2}) that
$$\lambda_{kD}(x)\leq\frac{1}{\mu_k}\max_{\sigma\in\Lambda}\sum_{s\in\mathscr B_{\sigma}}\lambda_{s}(x)+O(1),$$
where $\lambda_s(x)$ is the Weil function of $(s).$ Taking the expectation to get
$$km_f(r,D)\leq \frac{1}{\mu_k}\int_{\partial D(r)}\max_{\sigma\in\Lambda}\sum_{s\in\mathscr B_{\sigma}}\lambda_{s}(x)d\pi^r_o(x)+O(1).$$
Making use of Theorem \ref{xy12}, we obtain
$$km_f(r,D)
\leq \frac{\dim_{\mathbb C}V_k}{\mu_k}T_{f,kL}(r)+o\big(T_{f,kL}(r)\big)+O\Big(-\kappa(r)r^2+\log^+\log r\Big) \big\|.$$
This proves  Theorem \ref{thm1}.
\end{proof}

\noindent\textbf{Acknowledgement.} The author is   grateful to   
Prof. Min Ru and Dr. Yan He for their valuable suggestions on this paper.
\vskip\baselineskip

\label{lastpage-01}

\vskip\baselineskip
\vskip\baselineskip

\end{document}